\documentclass[11pt]{article}
\usepackage{my_preamble}
\usepackage{amssymb, amsmath, amsfonts}
\usepackage{fullpage}

\newcommand{\pto}{\rightharpoonup}
\newcommand{\F}{\mathbf{F}}

\newcommand{\G}{\mathcal{G}}

\renewcommand{\C}{\mathcal{C}}

\newtheorem*{obsv}{Observation}

\renewcommand{\rho}{\varrho}

\newcommand{\qto}{\rightsquigarrow}

\newcommand{\tpo}{\trianglelefteq}
\renewcommand{\F}{\mathbb{F}}

\begin{document}

\title{Generalized Indiscernibles as Model-complete Theories}
\author{
Cameron Donnay Hill\footnote{Correspondence to:   University of Notre Dame,
Department of Mathematics,
255 Hurley, Notre Dame, IN 46556. Email: \texttt{cameron.hill.136@nd.edu}}}
\maketitle

\begin{abstract}
We give an almost entirely model-theoretic account of both Ramsey classes of finite structures and of generalized indiscernibles as studied in special cases in (for example) \cite{kim-kim-scow-2012}, \cite{scow-2011}. We understand ``theories of indiscernibles'' to be special kinds of companionable theories of finite structures, and much of the work in our arguments is carried in the context of the model-companion. Among other things, this approach allows us to prove that the companion of a theory of indiscernibles whose ``base'' consists of the quantifier-free formulas is necessarily the theory of the \fraisse limit of a \fraisse class of linearly ordered finite structures (where the linear order will be at least quantifier-free definable). We also provide streamlined arguments for the result of \cite{kpt-2005} identifying extremely amenable groups with the automorphism groups of limits of Ramsey classes.
\end{abstract}

\section*{Introduction}

We give an almost entirely model-theoretic account of both Ramsey classes of finite structures and of generalized indiscernibles as studied in special cases in (for example) \cite{kim-kim-scow-2012}, \cite{scow-2011}. We take ``theories of indiscernibles'' to be special kinds of companionable Robinson theories $T_0$ in the sense of \cite{hru-1997} that also have a certain unconstrained modeling property (UMP), which formalizes the intuitive notion, ``what is needed to run a compactness argument that generates $T_0$-indiscernibles in a model.'' We show that this unconstrained modeling property is equivalent to a structural Ramsey property in the class of finite models of $T_0$, generalizing certain facts proved in \cite{kim-kim-scow-2012}, \cite{scow-2011}.

This approach allows us to prove, among other things, that the companion of a theory of indiscernibles whose ``base'' consists of the quantifier-free formulas is necessarily the theory of the \fraisse limit of a \fraisse class of linearly ordered finite structures (recapturing a result that was also proved, more or less, in \cite{kpt-2005} by different means). Our model-theoretic framework also allows us to restate and streamline the arguments of \cite{kpt-2005} -- demonstrating the equivalence of the Ramsey property and the \emph{extreme amenability} of the automorphism group of the \fraisse limit -- in more clearly model-theoretic terms; in particular, for the sufficiency of extreme amenability, we really only to consider the action of the automorphism group on certain type spaces (Stone spaces).

In this article, we present only a few simple examples to go with the technology -- leaving  applications to a forthcoming companion paper. In that paper, taking \cite{scow-2011} as inspiration, we mainly consider characterizations of classes of theories in the classification-theoretic hierarchy (NIP theories, NSOP theories, and so on) in terms of ``collapsing'' theories of indiscernibles to reduces.

\section{Definitions}

In this section, we collect almost all of the important definitions around ``theories of indiscernibles.'' The main concepts are Robinson theories and their model-companions, the Ramsey property and the modeling property for Robinson theories. In the next section, we establish how these ideas fit together. 

\subsection{Robinson theories}

\begin{defn}
Let $\L_0$ be some countable relational language. (Throughout, ``no function symbols and only finitely many constant symbols'' would also do.)
\begin{enumerate}
\item Wherever it appears, $\Delta$ denotes a set of $\L_0$-formulas which contains all of the quantifier-free $\L_0$-formulas, is closed under boolean combinations, under taking subformulas, and under substitutions of free variables. Under these conditions, $\Delta$ is called a \emph{base}.

When $\Delta$ is given, we define
$$\Sigma = \left\{\exists \xx\,\phi:\phi\in\Delta\right\},\,\,\Pi = \left\{\forall \xx\,\phi:\phi\in\Delta\right\}$$
Note that $\Delta$ may contain sentences.

\item For $\L_0$-structures $\A,\B$, a map $f:A\to B$ is a $\Delta$-embedding if for all $\aa\in A^n$ and $\phi(x_0,...,x_{n-1})\in\Delta$, $\A\models\phi(\aa)\iff\B\models\phi(f\aa)$.

\item An $\L_0$-theory $T_0$ is $\Delta$-universal if $T_0\subseteq\Pi$.
\end{enumerate}
\end{defn}

\begin{defn}[Robinson theories]
Let $\L_0$ be a countable relational language, and let $\Delta\subseteq\L_0$ be a base. We specify a \emph{Robinson theory over $\Delta$} to be a $\Delta$-universal theory $T_0$ such that:
\begin{enumerate}
\item For any $\B\models T_0$ and any $A_0\subseteq_\fin B$, there is a finite $\Delta$-submodel $\A\preceq_\Delta\B$ such that $A_0\subseteq A$.

\item For every $0<n<\omega$, the set $S^\Delta_n(T_0)$ of $\Delta$-complete $n$-types that are consistent with $T_0$ is finite.
\item (JEP) For all $\A_1,\A_2\models T_0$, there are $\B\models T_0$ and $\Delta$-embeddings $\A_i\to\B$, $i=1,2$.
\item (AP) For all $\A_0,\A_1,\A_2\models T_0$ and $\Delta$-embeddings $f_i:\A_0\to\A_i$, $i=1,2$, there are $\B\models T_0$ and $\Delta$-embeddings $g_i:\A_i\to\B$, $i=1,2$, such that $g_1f_1 = g_2f_2$. 
\end{enumerate}
A model $\A\models T_0$ will be called \emph{existentially closed} (e.c.) if for any $\Delta$-embedding $f:\A\to\B$, any $\phi(x_0,...,x_{n-1})\in\Sigma$, and any $\aa\in A^n$, $$\A\models\phi(\aa)\iff\B\models\phi(f\aa)$$
Define
$$T_0^* = \bigcap\left\{Th(\A):\A\models T_0\textrm{ is e.c.}\right\},$$
If $T_0^*$ is precisely the theory of e.c. models of $T_0$, then $T_0^*$ is called a $\Delta$-model-companion of $T_0$ (provided it is also a companion of $T_0$ with respect to $\Delta$-embeddings).  In fact, we can axiomatize $T_0^*$ more explicitly when it exists. For a finite model $\B\models T_0$, $B_0\subsetneq B$, and an enumeration $\bb_0\bb$ of $B$ with $B_0$ as an initial segment, let $\theta_{\B/B_0}(\xx,\yy)$ be the complete $\Delta$-type of $\bb_0\bb$, and let $\phi_{B/B_0}$ be the sentence $\forall \xx\exists\yy\theta_{\B/B_0}(\xx,\yy)$; let $\mathrm{Ext}(T_0)$ be the set of all such sentences. We will see now that a Robinson theory $T_0$ always has a $\Delta$-model-compantion.

\begin{obs}
Let $\A\models T_0$. Then $\A$ is e.c. if and only if $\A\models T_0\cup \mathrm{Ext}(T_0)$.
\end{obs}

\begin{prop}
For every model $\A\models T_0$, there is an e.c. model $\A^*\models T_0$ such that $\A\preceq_\Delta\A^*$.
\end{prop}
\begin{proof}[Proof (Sketch)]
For a model $\A\models T_0$, let $K(\A)$ be the following set of sentences (in the expansion of $\L_0$ in which every element of $A$ is named as a new constant symbol):
$$T_0\cup \diag_\Delta(\A)\cup\left\{\exists \yy\,\theta_{\B/B_0}(\bb_0,\yy):\B\models T_0\textnormal{ finite}, B_0\subsetneq B\right\}$$
where the $\Delta$-formula $\theta_{\B/B_0}$ is just as defined above. From the assumptions that the class of models of $T_0$ has AP and JEP with respect to $\Delta$-embeddings, one finds that $K(\A)$ has a model, say $\hat \A$. Clearly, $\A\preceq_\Delta\hat\A\r\L_0$, and $\hat\A\r\L_0\models\exists \yy\,\theta_{\B/B_0}(\bb_0,\yy)$ whenever $\bb_0\in A^{|B_0|}$, $\B\models T_0$ is finite, and $\tp_\Delta^\B(\bb_0)=\tp_\Delta^\A(\bb_0)$. 

Now, as usual, we think of $\A\mapsto \A':=\hat\A\r\L_0$ as an operator on models of $T_0$. We construct a chain of models
$$\A_0\preceq_\Delta\A_1\preceq_\Delta\cdots\preceq_\Delta\A_n\preceq_\Delta\cdots$$
(the chain being of order-type $\omega$) by taking $\A_0=\A$ and, for $n<\omega$, $\A_{n+1} = \A_n'$. From the assumptions that $\Delta$ is closed under boolean combinations and under taking subformulas, it is not hard to see that $\A_m\preceq_\Delta\A^*$ for all $m<\omega$ and $\A^*\models T_0$, where $\A^*=\bigcup_n\A_n$. Moreover, $\A^*\models T_0\cup\mathrm{Ext}(T_0)$, so $\A^*$ is existentially closed.
\end{proof}


We also note that the class of models of $T_0$ is a concrete category with $\Delta$-embeddings for morphisms. When this understanding is in play, and when $\A,\B$ are models of $T_0$, we write $\hom(\A,\B)$ for the set of $\Delta$-embeddings $\A\to\B$. We may write $\hom_\Delta(\A,\B)$ if there is some risk of ambiguity.
\end{defn}

We state a few satisfying facts about Robinson theories in our sense. Their proofs are quite routine, so we omit them.

\begin{fact}
Assume $T_0$ is a Robinson theory. For every finite model $\A\models T_0$, there is a sentence $\phi^\A\in\Sigma$ such that for all $\B\models T_0$, $\B\models\phi^\A$ if and only if there is a $\Delta$-embedding $\A\to\B$.
\end{fact}

\begin{fact}
Let $T_0$ be a Robinson theory over $\Delta\subseteq\L_0$. For every formula $\phi(\xx)\in\L_0$, there are $\psi_\phi(\xx)\in\Sigma$ and $\psi'_\phi(\xx)\in\Pi$ such that
$$T_0^*\models\forall\xx\left(\phi(\xx)\bic\psi_\phi(\xx)\bic\psi'_\phi(\xx)\right)$$
\end{fact}

\begin{fact}
Let $T_0$ be a Robinson theory over $\Delta\subseteq\L_0$. For any model $\A\models T_0^*$, $\diag_\Delta(\A)\cup T_0^*$ is consistent.
\end{fact}

Our last fact is an immediate consequence of the fact that the class of models of $T_0$ has JEP and AP with respect to $\Delta$-embeddings.

\begin{fact}
Let $T_0$ be a Robinson theory over $\Delta\subseteq\L_0$. Then $T_0^*$ is the $\Delta$-model-{\em completion} of $T_0$, and it is a complete theory.
\end{fact}

\subsection{Modeling and the Ramsey property}

We now establish the two properties that could reasonably make a Robinson theory into a theory of indiscernibles. The first of these -- the Ramsey property -- is essentially a combinatorial partition property of the class of finite models of a Robinson theory $T_0$. The second -- the modeling property -- is intended to be just what is needed in order to use the Compactness Theorem to generate an indiscernible model of $T_0^*$, which we call an indiscernible $T_0$-picture, inside a model of a given $\L$-theory $T$. First, we need to fix a relatively weak notion of embedding that is more appropriate for accomodating the Ramsey property. (We will also see later in the article that a certain stronger Ramsey property with ``standard'' $\Delta$-embeddings is not possible.)

\begin{defn}
Suppose $\A$ and $\B$ are $\L_0$-structures. An $\ell$-embedding $f:\A\qto\B$ consists of a system $(\A_X,f_X)_{X\subset_\fin A}$ where for each $X\subset_\fin A$, $\A_X$ is a finite $\Delta$-sub-model of $\A\models T_0$ such that $X\subseteq A_X$, and $f_X:\A_X\to \B$ is a $\Delta$-embedding; we also require that there is a (usually unmentioned) function $q:\omega\to\omega$ such that $|A_X|\leq q(|X|)$ for all $X\subset_\fin A$. (The $\ell$ stands for ``\underline{$\ell$}ocalized embedding.'')

Now, suppose $f:\A\qto\B$, $g:\B\qto\C$ are $\ell$-embeddings. We define an $\ell$-embedding $g{\circ}f:\A\qto\C$ as follows: For each $X\subset_\fin A$, we set $A^{g{\circ}f}_X=A_X$ and $(g{\circ}f)_X=g_{f_X[A_X]}\circ f_X$.
\end{defn}

\begin{defn}[Ramsey property]
Let $T_0$ be a Robinson theory over $\Delta$ in $\L_0$. We say that $T_0$ has the Ramsey property if the following holds:
\begin{quote}
Let $\B\models T_0^*$, and let $A$ be a finite model of $T_0$. For some $k<\omega$, let $h:\hom(A,\B)\to k$ be a finite coloring. Then there is an $\ell$-embedding $f:\B\qto\B$ such that $h$ is constant on $\hom(A,f_X[B_X])$ for every $X\subset_\fin B$.
\end{quote}
\end{defn}


In order to state the modeling property, we must first specify what we mean by a $T_0$-picture. It is also essential to provide some notion of EM-type for models of $T_0^*$; for this we define the notion of $T_0$-template. The forms of our definitions here are loosely inspired by the notion of coherent sequences in \cite{hyttinen-2000}.
\begin{defn}
Let $T_0$ be a Robinson theory over $\Delta$ in $\L_0$. Let $\L$ be some language, and let $\M$ be an $\L$-structure.
\begin{itemize}
\item A {\em $T_0$-picture in $\M$} is a pair $(\A,g)$ where $\A\models T_0$ and $g:A\to M$ is a one-to-one function. 

\item Let $(\A,g)$ be a $T_0$-picture in $\M$, and let $C\subset M$. We say that $(\A,g)$ is $C$-indiscernible if for all $n<\omega$ and all $\aa,\aa'\in A^n$, 
$$\tp_\Delta^\A(\aa)=\tp_\Delta^\A(\aa')\,\,\implies\,\,\tp^\M(g\aa/C)= \tp^\M(g\aa'/C)$$

\item A {\em $T_0$-template in $\M$} is a pair $(\B,\Gamma)$ where $\B\models T_0$ and $\Gamma$ is a map $B^{<\omega}\to M^{<\omega}$ satisfying:
\begin{enumerate}
\item $(\forall n<\omega)(\forall \bb\in B^{n})\bigwedge\begin{cases}
\Gamma(\bb)\in M^n\\
\qtp_=^\B(\bb) = \qtp^\M_=(\Gamma(\bb))
\end{cases}$

\item $\Gamma$ respects permutations of coordinates: 

For any $(b_0,...,b_{n-1})\in B^n$ and $\sigma\in \emph{Sym}(n)$, if $\Gamma(\bb) = (c_0,...,c_{n-1})$, then $\Gamma(b_{\sigma(0)},...,b_{\sigma(n-1)}) = (c_{\sigma(0)},...,c_{\sigma(n-1)})$.

\item For all $\aa,\bb\in B^{<\omega}$,
$\tp^\M(\Gamma(\aa\widehat{\,\,}\bb)) = \tp^\M(\Gamma(\aa)\widehat{\,\,}\Gamma(\bb))$

\end{enumerate}
(Here, $\qtp_=$ indicates a quantifier-free-complete type in the language of equality.)

\item Let $(\A,g)$ be a $T_0$-picture in $\M$, and let $(\B,\Gamma)$ be a $T_0$-template in $\M_1$.  We say that {\em $(\A,g)$ is patterned on $(\B,\Gamma)$} if there is a family of patterning maps $f = (f^F)_{F\subset_\fin\L}$ -- meaning that for each $F\subset_\fin\L$, $f^F:\A\qto\B$ is an $\ell$-embedding  such that 
$$\tp^{\M_1}_F(g\aa) = \tp^\M_F(\Gamma(f^F_X\aa))$$
for all $n<\omega$,  $X\subset_\fin M$, and $\aa\in A^n_X$. 
\end{itemize}
\end{defn}

\begin{defn}[Modeling property]
Let $T_0$ be a Robinson theory over $\Delta$ in $\L_0$. Let $\L$ be some language, and let $T$ be a complete $\L$-theory. We say that $T_0$ has the modeling property with respect to $T$ if the following holds:
\begin{quote}
Suppose $(\A,\Gamma)$ is a $T_0$-pattern in $\M\models T$, where $\A\models T_0^*$. Then there are an $\M_1\models T$ and an indiscernible $T_0$-picture $(\A,g)$ in $\M_1$ patterned on $(\A,\Gamma)$.
\end{quote}
We say that $T_0$ has the {\em unconstrained modeling property} (UMP) if it has the modeling property with respect to every complete theory in any {\em countable} language.
\end{defn}

The terminology ``modeling property'' is, to the best of our knowledge, due to L. Scow.

\section{Results on modeling and the Ramsey property}

Now that we have established the relevant definitions, we will determine in this section just how they fit together. In the first subsection, we show that the Ramsey property and the unconstrained model property (together with ``fintie rigidity'') are the same thing. In much of the literature on structural Ramsey theory, only \fraisse classes of \emph{linearly ordered} finite structures are considered; to some extent, this appears to be just a convenient way of ensuring finite-rigidity. (In general, rigidity does not imply that there is a uniformly definable linear order in a class of finite structures.)  As we show in the succeeding subsection, this not in fact just a convenience -- any theory of indiscernibles with base $\Delta$ has a $\Delta$-definable linear order which is ``almost dense.''

\subsection{Basic equivalences}

The proof of the next lemma was suggested by L. Scow. 

\begin{lemma}\label{lemma:rp-to-frp}
Let $T_0$ be a Robinson theory over $\Delta$ in $\L_0$. If $T_0$ has the Ramsey property, then $T_0$ has the following {\em finitary Ramsey property}:
\begin{quote}
Let $\A,\B\models T_0$ be finite, and let $0<k<\omega$. Then there is a finite model $\C\models T_0$ such that for any coloring $h:\hom(\A,\C)\to k$, there is an $e_0\in\hom(\B,\C)$ such that $h(e_0f)=h(e_0g)$ for all $f,g\in\hom(\A,\B)$. (That is, $h$ is constant on $\hom(\A,e_0[\B])$.)
\end{quote}
\end{lemma}
\begin{proof} 
In this proof, it will be convenient to fix a listing $\D_0,\D_1,...,\D_n,...$ of all finite models of $T_0$ up to isomorphism. (This is possible because of the finiteness of $S_n^\Delta(T_0)$ for every $n<\omega$.)

Let $\A,\B\models T_0$ be finite models, and (w.l.o.g.) assume that $\hom(\A,\B)$ is nonempty. Let $0<k<\omega$, and towards a contradiction, suppose that for every finite $\C\models T_0$, there is a coloring $h^\C:\hom(\A,\C)\to k$ such that for every $e\in \hom(\B,\C)$, $h^\C$ is \emph{not} constant on $\hom(\A, e\B)$. Fix an order $<_A$ of $A$. For each finite $\C\models T_0$, we define an expanded structure $(\C,h_\C)$ as follows\footnote{The notation here is so preposterously awful, we couldn't resist: Cf. $\frac{\,\,\overline{\Xi}\,\,}{\Xi}$ for a complex number $\Xi\neq0$.}:
\begin{enumerate}
\item $(\C,h^\C)$ has an additional sort $X$ (in addition to the home-sort $X_0$), a new function symbol $h:X_0^n\to X$, and new constant symbols $c_*,c_0,...,c_{k-1}$ of sort $X$.
\item $X^{(\C,h^\C)}=\{\star,0,1,...,k-1\}$, $c_*^{(\C,h^\C)} = \star$, and $c_i^{(\C,h^\C)} = i$ for each $i<k$
\item For $c'_0,...,c'_{n-1}$, if there is an $e\in \hom(\A,\C)$ such that $e(a_i) = b_i$ for each $i<n$, then $h^{(\C,h^\C)}(\cc')= h^\C(e)<k$. Otherwise, $h^{(\C,h^\C)}(\cc') = \star$.
\end{enumerate}
Clearly, if $\C_1,\C_2\models T_0$ are finite and $f:\C_1\to\C_2$ is a $\Delta$-embedding, then $f$ extends uniquely to an $\Delta$-embedding $f:(\C_1,h^{\C_1})\to (\C_2,h^{\C_2})$. Assume $B = \{b_0,...,b_{N-1}\}$, and let $\theta_\B(x_0,...,x_{N-1})$ be a formula such that if $\C\models T_0$ and $\C\models \theta_\B(\cc)$, then $(b_i\mapsto c_i)_{i<N}\in\hom(\B,\C)$. (This is possible because $S_n^\Delta(T_0)$ is finite for every $n<\omega$.) Define $\theta_\A(y_0,...,y_{n-1})$ similarly.  Let $\phi_{bad}$ be the sentence,
$$\forall x_0,...,x_{N-1}\left(\theta_\B(\xx)\cond \bigvee_{\yy,\zz\subset \xx}\theta_\A(\yy)\wedge\theta_\A(\zz) \wedge h(\yy)\neq h(\zz)\right)$$
Hence $(\C,h^\C)\models \phi_{bad}$ for all finite $\C\models T_0$.

 We may then choose a sequence of finite models $\C_n\models T_0$ ($n<\omega$) such that for all $m\leq n<\omega$, $\C_m\preceq_\Delta\C_n$, and for all $n<\omega$, $\C_n\models\bigwedge_{i<n}\phi^{\D_i}$, where for each $i$, $\phi^{\D_i}$ is a sentence of $\Sigma$ such that $\models\phi^{\D_i}$ is equivalent to containing an isomorphic copy of $\D_i$ is a $\Delta$-substructure. Let $\Psi$ be a non-principal ultrafilter on $\omega$. Clearly,
$$(\M_0,h^{\M_0}) = \prod_n(\C_n,h^{\C_n})/\Psi\models T_0\cup\{\phi_{bad}\}$$
and $(\M_0,h^{\M_0})\models \phi^{\D_i}$ for all $i<\omega$. Now, let $\M\models T_0^*$ such that $\M_0\preceq_\Delta\M$, and let $h^\M:\hom(\A,\M)\to k$ be a coloring extending $h^{\M_0}$ such that for any finite $\C\preceq_\Delta\M$ satisfying $T_0$, there is a $\Delta$-embedding $g:\C\to\M_0$ such that $h^\M(e) = h^{\M_0}(ge)$ for all $e\in\hom(\A,\C)$. (This is possible just by compactness.)

As $T_0$ has the Ramsey property, there is an $\ell$-embedding $f:\M\qto\M$ such that $h$ is constant on $\hom(\A,f_X[M_X])$ for all $X\subset_\fin M$. We may assume that $\B$ appears as a $\Delta$-substructure of $\M$, and in this case, $h^\M$ is constant on $\hom(\A,f_B[M_B])$. By construction, there is a $\Delta$-embedding $g:f_B[M_B]\to\M_0$ such that $h^\M(e) = h^{\M_0}(ge)$ for all $e\in\hom(\A,f_B[M_B])$. It follows that $\M_0\models\neg\phi_{bad}$, a contradiction. This completes the proof.
\end{proof}

\begin{lemma}\label{lemma:frp-to-rp}
Let $T_0$ be a Robinson theory over $\Delta$ in $\L_0$. If $T_0$ has the finitary Ramsey property, then $T_0$ has the Ramsey property.
\end{lemma}
\begin{proof}
Let $\M\models T_0^*$, $\A\models T_0$ finite, and $h:\hom(\A,\M)\to k$ be given. We must show that there is an $\ell$-embedding $f:\M\qto\M$ such that $h$ is constant on $\hom(\A,f_X[M_X])$ whenever $X\subset_\fin M$. By definition of a Robinson theory, for each $X\subset_\fin M$, let $M_X\prec_\Delta\M$ be an finite model of $T_0$ containing $X$. By the finitary Ramsey property, for each $X\subset_\fin M$, there is a finite model $\C_X\models T_0$ extending $M_X$ such that for any coloring $w:\hom(\A,\C_X)\to k$, there is a $\Delta$-embedding $f_X:M_X\to \C_X$ such that $w(f_Xe) = w(f_Xe')$ for all $e,e'\in\hom(\A,M_X)$. Since $\M$ is existentially closed, we may assume that $\C_X$ is a $\Delta$-substructure of $\M$, so the mappings $f_X:M_X\to\C_X\prec_\Delta\M$ suffice for the required $\ell$-embedding.
\end{proof}
\begin{thm}
Let $T_0$ be a Robinson theory over $\Delta$ in $\L_0$. Then $T_0$ has the Ramsey property if and only if  $T_0$ has the finitary Ramsey property.
\end{thm}
\begin{proof}
Lemmas \ref{lemma:rp-to-frp} and \ref{lemma:frp-to-rp}. 
\end{proof}

\begin{lemma}[Rigidity]\label{lemma:finite-rigidity}
Let $T_0$ be a Robinson theory over $\Delta$ in $\L_0$. If $T_0$ has the Ramsey property, then $T_0$ is {\em finitely-rigid}: For all  $\A\models T_0$ finite, $\hom(\A,\A)=\{1_A\}$.
\end{lemma}
\begin{proof}
Towards a contradiction, suppose $\A\models T_0$ is finite, and let $\sigma\in\hom(\A,\A)$ be non-trivial. (Obviously, $\sigma$ must be an automorphism of $\A$.) Let $\M\models T_0^*$, and let $h:\hom(\A,\M)\to 2$ such that for any $\Delta$-embedding $e:\A\to\M$,  $h(e\sigma)\neq h(e)$. Now, suppose $f:\M\qto\M$ is an $\ell$-embedding such that $h$ is constant on $\hom(\A,f_X[M_X])$ for every $X\subset_\fin M$. Fix an embedding $e_0:\A\to\M$ and let $X = e_0[A]$. For $e\in \hom(\A,f_X[M_X])$, $f_Xe, (f_Xe)\sigma\in\hom(\A,\M)$ are distinct, and $h\left(fe\right)\neq h\left((fe)\sigma\right)$, a contradiction. 
\end{proof}

\begin{thm}\label{thm:main-equivalence}
Let $T_0$ be a Robinson theory over $\Delta$ in $\L_0$. Then $T_0$ has the Ramsey property if and only if $T_0$ is finitely-rigid and has the unconstrained modeling property. 
\end{thm}

\begin{proof}[Proof: ``if'']
Assume $T_0$ is finitely-rigid and has UMP. Let $\M$ be an e.c. model of $T_0$, and let $\A$ be a finite model of $T_0$. Let $h:\hom(\A,\M)\to k$ be a coloring. We define a language $\L\supset\L_0$ with new $n$-ary predicates $P_0,...,P_{k-1}$, where $n=|A|$. Let $\M'$ be the $\L$-expansion of $\M$ with the following interpretations: Let $(a_0,...,a_{n-1})\in M^n$; if $f\in\hom(A,\M)$ with $f[A] = \{a_i\}_{i<n}$ and $h(f)=i$, then $\aa\in P_i^{\M'}$. By finite-rigidity, this is well-defined.
Let $T = Th(\M')$, and consider the $T_0$-template $(\M,\Gamma)$ in $\M'$ where $\Gamma$ is the identity map $M^{<\omega}\to M^{<\omega}$. By UMP, let $(\M,g)$ be an indiscernible $T_0$-picture in some $\N'\models T$ patterned on $(\M,\Gamma)$, and let $(f^F:\M\qto\M)_{F\subset_\fin\L}$, be the family of patterning maps. Fix $F = \{P_0(\xx),...,P_{k-1}(\xx)\}$

We claim that $h$ is constant on $\hom(A, f^F_X[M_X])$ for every $X\subset_\fin M$. Let $e,e'\in\hom(A,f_X^F[M_X^F])$. Then, let $\bb,\cc\in M^n$  be enumerations of $e[A]$ and $e'[A]$, respectively, such that $\tp^\M_\Delta(\bb) = \tp^\M_\Delta(\cc)$. By indiscernibility, $\tp^{\N'}(g\bb) = \tp^{\N'}(g\cc)$, so there is a (unique) $i<k$ such that $\N'\models P_i(g\bb)\wedge P_i(g\cc)$. By definition of patterning and our choice $\Gamma$, 
$$\tp^{\M'}_F(f^F_X\bb) = \tp^{\M'}_F(\Gamma(f^F_X\bb)) = \tp^{\N'}_F(g\bb)$$
$$\tp^{\M'}_F(f^F_X\cc) = \tp^{\M'}_\Delta(\Gamma(f^F_X\cc)) = \tp^{\N'}_F(g\cc)$$
and it follows that $h\left(f^F_Xe\right) = h\left(f^F_Xe'\right) = i$. Thus, $f$ witnesses the Ramsey property requirement given by $h$.
\end{proof}

\begin{proof}[Proof: ``only if'']
Let $T$ be a complete $\L$-theory, and let $(\A,\Gamma)$ be a $T_0$-pattern in $\M\models T$, where $\A$ is an e.c. model of $T_0$. 
 Let $0<n<\omega$ and $F\subset_\fin\L$. By Lemma \ref{lemma:finite-rigidity}, $T_0$ is finitely-rigid. Let $\bb_0,...,\bb_{N-1}\in A^n$ represent all $\Delta$-complete $n$-types over $\emptyset$ realized in $\A$, and for each $i<N$, let $B_i = \A\r rng(\bb_i)$.  Without loss of generality, we assume that each $\bb_i$ enumerates a finite model of $T_0$ -- if not we can extend them to finite models and make very minor adjustments to the argument that follows.
 
 \bigskip\noindent\textbf{Construction}$(n,F, id^A)$:
 Define $f^F:\A\qto\A$ as follows:
\begin{enumerate}
\item Define $h:\hom(B_0,\A)\to  S_n^F(T)$ by 
$$e\mapsto \tp^\M_F(\Gamma(e\bb_0)) = \tp^\M_F(\Gamma(\,(id^A e)\bb_0))$$
\item By the Ramsey property, let $\phi^0:\A\qto\A$ be an $\ell$-embedding such that $h$ is constant on $\hom(B_0,\phi^0_X[A_X])$ for all $X\subset_\fin A$.
\end{enumerate}
For $0<i<N-1$, given $\phi^0,...,\phi^{i-1}:\A\qto\A$, we define $\phi^i$ as follows:
\begin{enumerate}
\item[1i.] Define $h:\hom(B_i,\phi^{i-1}\A)\to  S_n^F(T)$ by 
$$e\mapsto \tp^\M_F(\Gamma(e\bb_i)) = \tp^\M_F(\Gamma(\,(id^Ae)\bb_i))$$

\item[2i.] By the Ramsey property, let $\phi^i:\phi^{i-1}\A\qto \phi^{i-1}\A$ be a $\Delta$-embedding such that $h$ is constant on $\hom(B_i,\phi^i\phi^{i-1}\A)$.
\end{enumerate}
Finally, we set $f^F = \phi^{N-1}\circ\cdots\circ\phi^1\circ\phi^0$

\bigskip
\begin{obsv}
For each $i<N$, $X\subset_\fin A$ and $e,e'\in\hom(B_i,f^F_X[A_X])$, 
$$\tp^F(\Gamma(f^Fe\bb_i)) = \tp^F(\Gamma(f^Fe'\bb_i))$$
\end{obsv}

 Note that if $v:\A\qto\A$ is some $\ell$-embedding, we can start the above construction with $v$ in place of $id^A$, and we would denote this by \textbf{Construction}$(n,F, v)$
Now, let $F_n\subset_\fin\L$ ($n<\omega$) be such that $F_n\subset F_{n+1}$ and $\bigcup_nF_n = \L_0$.

\begin{itemize}
\item Let $\eta^0$ be the result of \textbf{Construction}$(1,F_0, id_A)$;
\item Given $\eta^{n-1}$, let $\eta^n$ be the result of \textbf{Construction}$(n+1,F_n, \eta^{n-1})$
\end{itemize}

\bigskip
In the remainder of the argument, we use the forgoing construction to run a compactness argument that actually generates the required indiscernible $T_0$-picture. Let $L(\L/\L_0)$ be the language with two sorts $X_0$ and $X_\L$, all of the symbols of $\L$ on the sort $X_\L$, all the symbols of $\L_0$ on $X_0$, constant symbols $c_a$ on $X_0$ for each $a\in A$, and a function symbol $g:X_0\to X_\L$. We now define theory that -- just to give it a convenient name -- we call the Scow-theory of $\A$ in $T$ -- it is very similar to a certain theory defined in \cite{scow-2011}.
Let $\mathtt{Scow} = \mathtt{Scow}(\A,T)$ be the theory asserting the following:
\begin{itemize}
\item $X_\L\models T$ and $X_0\models T_0^*$
\item $\left\{\theta(c_{a_0},...,c_{a_{n-1}}): \theta(a_0,...,a_{n-1})\in\diag_\Delta(\A)\right\}$
\item ``$g$ is one-to-one.''
\item For each $0<n<\omega$, $q(x_0,...,x_{n-1})\in S_n^\Delta(T_0)$, $a_0,...,a_{n-1},b_0,....,b_{n-1}\in A$ and  $\phi(y_0,...,y_{n-1})\in\L$, 
$$q(\cc_a)\wedge q(\cc_b)\cond (\phi(g\cc_a)\bic \phi(g\cc_b))$$
where $\cc_a = (c_{a_0},...,c_{a_{n-1}})$ and $\cc_b = (c_{b_0},...,c_{b_{n-1}})$.

\item For each $0<n<\omega$ and $F\subset_\fin\L$, 
$$\bigvee_{\nu\in \Phi_n(\Gamma,e)}\left[\bigwedge_{q\in S^\Delta_n(T_0)}\forall \xx\in X_0^n\left(q(\xx)\cond \nu_q(g\xx)\right)\right] $$
where $S^F_n(\Gamma)$ is the set of $n$-$F$-types $\pi(\xx)$ such that $\Gamma(\aa)\models\pi$ for some $\aa\in A^n$, and $\Phi_n(\Gamma,e)$ is the set of functions $\nu:S^\Delta_n(T_0)\to S^F_n(\Gamma)$ defined as follows: Let $k$ be the smallest number such that $F\subseteq F_k$; then $\nu_q=\tp^\M_F(\Gamma(\eta^k\aa))$ where $\aa\in A^n$ and $\A\models q(\aa)$.
\end{itemize}
From our prior constructions, it's not hard to see that this Scow-theory is finitely satisfiable, and in a model $\MM_1 = (\M_1,\A_1)\models \mathtt{Scow}$, the substructure on $\{c_a^{\MM_1}\}_{a\in A}$ is isomorphic to $\A$, and the interpretation $g^{\MM_1}$ yields an indiscernible $T_0$-picture patterned on $(\A,\Gamma)$ with $\eta^k$'s as patterning maps.
\end{proof}

With these equivalences in place, it's now natural to define (finally!) what we mean by a ``theory of indiscernibles'':

\begin{defn}
A {\em theory of (generalized) indiscernibles} is a Robinson theory $T_0$  that has the Ramsey property.
\end{defn}


\subsection{Rigidity and order}

As promised, in this last subsection, we investigate the finite-rigidity condition in some more depth. In particular, we will see that a theory of indiscernibles must have a 0-definable linear order -- indeed, a $\Delta$-definable linear order -- so adding a linear order to the language need not provide any new powers. In a rather different manner, this fact was first proved in \cite{kpt-2005} in the quantifier-free case and only recovering a 0-definable linear order. Our result is slightly finer, and our demonstration differs enough, we think, to be interesting in itself.

\begin{defn}[Irreflexive types]
Let $p(x_0,...,x_{n-1})\in S^\Delta_n(T_0)$. We say that $p$ is irreflexive if $T_0\cup p\models \bigwedge_{i<j<n}x_i\neq x_j$.
\end{defn}

\begin{lemma}
Let $T_0$ be a Robinson theory, and suppose $T_0$ has the modeling property with respect to an $\L$-theory $T$ whose models are (expansions of) infinite linear orders. Then for any irreflexive 2-type $p(x,y)\in S_2^\Delta(T_0)$, 
\begin{enumerate}
\item $T_0\cup p(x,y)\cup p(y,x)$ is inconsistent;
\item $T_0\cup p(x,y)\cup p(y,z)\cup p(z,x)$ is inconsistent;
\end{enumerate}
\end{lemma}
\begin{proof}
We prove the lemma with $T = DLO$, but this imposes no real loss of generality. With $p(x,y)$ as described, suppose $T_0\cup p(x,y)\cup p(y,x)$ is consistent. Then there are an e.c. model $\A$ of $T_0$ and $a_1,a_2\in A$ such that $\A\models p(a_1,a_2)$ and $\A\models p(a_2,a_1)$. Let $(\A,\Gamma)$ be any $T_0$-template with $\Gamma:A^{<\omega}\to\QQ^{<\omega}$. By the modeling property with respect to DLO, let $(\QQ^*,<)\models DLO$, and let $(\A,g)$ be an indiscernible $T_0$-picture in $\QQ^*$ patterned on $(\A,\Gamma)$. Without loss of generality, suppose $g(a_1)<g(a_2)$; then since $\qtp^\A(a_1,a_2) = p = \qtp^\A(a_2,a_1)$, it must be that $g(a_2)<g(a_1)$ as well, which is impossible. Thus, $T_0\cup p(x,y)\cup p(y,x)$ must be inconsistent.

Similarly, suppose  $T_0\cup p(x,y)\cup p(y,z)\cup p(z,x)$ is consistent. Then there are an e.c. model $\A$ of $T_0$ and $a,b,c\in A$ such that $\A\models p(a,b)$, $\A\models p(b,c)$ and $\A\models p(c,a)$. Let $(\A,\Gamma)$ be any $T_0$-template with $\Gamma:A^{<\omega}\to\QQ^{<\omega}$. By the modeling property with respect to DLO, let $(\QQ^*,<)\models DLO$, and let $(\A,g)$ be an indiscernible $T_0$-picture in $\QQ^*$ patterned on $(\A,\Gamma)$. First, suppose $g(a)<g(b)$; then since $(b,c)\models p$ and $(c,a)\models p$, we find that $g(b)<g(c)$ and $g(c)<g(a)$, contradicting the transitivity and anti-symmetry of $<$ in $\QQ^*$. Similarly, if $g(b)<g(a)$, we would have $g(c)<g(b)$ and $g(a)<g(c)$ -- again contradicting transitivity and anti-symmetry. 
\end{proof}

\begin{prop}\label{prop:sop-to-order}
Let $T_0$ be a theory of indiscernibles. Then $T_0$ has the following pervasive form of the strict-order property: 
For any model $\A\models T_0$ and any infinite set $X\subseteq A$, there are an irreflexive 2-type $p(x,y)\in S_2^\Delta(T_0)$ and infinite subset $X_0\subseteq X$ such that $p$ linearly orders $X_0$.
\end{prop}
\begin{proof}
Let $\A\models T_0$, and let $X\subseteq A$ be infinite. Let $\prec$ be any linear order of $X$.  Define $f:{X\choose 2}\to S_2^\Delta(T_0)$ by $\{a\prec b\}\mapsto \qtp^\A(a,b)$. By the classical Ramsey theorem, there are an infinite subset $\Lambda\subseteq X$ and a 2-type $p(x,y)\in S^\Delta_2(T_0)$ such that $\A\models p(a,b)$ whenever $a\prec b$ and $a,b\in \Lambda$. Note that $p(x,y)$, being $\Delta$-complete, is irreflexive. By the previous lemma, $p(x,y)$ defines the linear order $\prec$ on $\Lambda$, and this suffices.
\end{proof}

\begin{thm}
Let $T_0$ be a theory of indiscernibles. Then there is a $\Delta$-formula $\phi(x_1,x_2)$ that defines a linear order of the universe in every model of $T_0$.
\end{thm}
\begin{proof}
For each $0<n<\omega$ and $W\subset S_2^\Delta(T_0)$, let $\theta_W^n(x_0,...,x_{n-1})$ be the formula asserting that $<_W$ is a linear order of the set $\{x_0,...,x_{n-1}\}$, where 
$$x_1<_Wx_2\,\,\iff\,\,\bigvee_{p\in W}p(x_1,x_2)$$
The next claim follows immediately from Proposition \ref{prop:sop-to-order}.

\begin{claim}
Let $\{c_i\}_{i<\omega}$ be a set of new constant symbols, and let $\Psi$ be the following set of sentences in $\L_0$ expanded with these new symbols:
$$T_0^*\cup\{c_i\neq c_j\}_{i<j<\omega}\cup \bigcup_{i_0<\cdots<i_{n-1}<\omega}\bigcup_{W\subset S_2^\Delta(T_0)}\left\{\neg\theta_W^n(c_{i_0},...,c_{i_{n-1}})\right\}$$
If $\Psi$ is consistent, then $T_0$ does not have the Ramsey property.
\end{claim}
Since $T_0$ {\em does} have the Ramsey property, there must be an $N<\omega$  such that for any $N\leq n<\omega$, 
$$T_0\models \forall x_0...x_{n-1}\left(\bigwedge_{i<j<n}x_i\neq x_j\cond \bigvee_{W\subset S_2^\Delta(T_0)}\theta^n_W(\xx)\right)$$
Now, let $\A$ be a countable model of $T_0^*$, and for each $n<\omega$, let $A_n\subset_\fin A$ such that $A_n\subset A_{n+1}$ and $\bigcup_nA_n = A$.
For each $n<\omega$, choose $W_n\subset S_2^\Delta(T_0)$ such that $<_{W_n}$ is a linear order of $A_n$. Then, by the Pigeonhole Principle, there is a $W\subset S_2^\Delta(T_0)$ such that $W_n=W$ for infinitely many $n<\omega$. It follows that $\phi(x_1,x_2) = \bigvee_{p\in W}p(x_1,x_2)$ is a linear order of $A$, so 
$$\A\models \bigwedge\begin{cases}
\forall x_1x_2\left(\phi(x_1,x_2)\cond x_1\neq x_2\wedge\phi(x_2,x_1)\right)\\
\forall x_1x_2\left(\phi(x_1,x_2)\vee x_1=x_2\vee\phi(x_2,x_1)\right)\\
\forall x_1x_2x_3\left(\phi(x_1,x_2)\wedge\phi(x_2,x_3)\cond \phi(x_1,x_3)\right)
\end{cases}$$
Since $T_0^*$ is a complete theory, $T_0^*$ implies the conjunction on the right. Moreover, this conjunction is universal (over $\Delta$), so since every model of $T_0$ is a $\Delta$-submodel of a model of $T_0^*$, $\phi(x_1,x_2)$ is a linear order of the universe of any model of $T_0$.
\end{proof}

\begin{defn}
An ordered theory of indiscernibels is just a theory of indiscernibles $T_0$ in a language $\L_0=\{<,...\}$ such that $<$ defines a linear order of the universe in every model of $T_0$.
\end{defn}

\begin{cor}
Let $T_0$ be a theory of indiscernibles. Then $T_0$ has a $\Delta$-definitional expansion $T_0'$ in a language $\L_0'=\L_0\dot\cup\{<\}$ which is an ordered theory of indiscernibles.
\end{cor}

\subsection{Connections to traditional Ramsey and \fraisse classes}

In this subsection, we complete our demonstration that the ``standard practice'' in structural Ramsey theory of working with \fraisse classes of ordered structures is the ``correct'' way to proceed. More precisely, we will see that in the case where $T_0$ is a theory of indiscernibles whose base $\Delta$ consists of just the quantifier-free formulas, $T_0^*$ \emph{must be} the theory of the \fraisse limit of a \fraisse class of linearly-ordered finite structures.

To conclude this subsection (and this section), we will look a little closer at the kind of linear order that can arise in a theory of indiscernibles -- showing that it must be ``almost dense'' in a certain sense. We then use this characterization to show that our definition of the Ramsey property in terms of $\ell$-embeddings is necessary insofar as a Ramsey property defined in terms of $\Delta$-embeddings between models of $T_0^*$ is not possible.

\subsubsection{Homogeneity and quantifier elimination}

\begin{lemma}
Let $T_0$ be an ordered theory of indiscernibles in a lanague $\L_0=\{<,...\}$, and let $0<n<\omega$ and $\M\models T_0^*$. Suppose $\bb_1,\bb_2\in M^n$ satisfy $\bigwedge_{i<j<n}x_i<x_j$. Then,
$$\tp_\Delta(\bb_1) = \tp_\Delta(\bb_2)\Longrightarrow\begin{cases} 
\tp_\Sigma(\bb_1)=\tp_\Sigma(\bb_2)\\
\tp_\Pi(\bb_1)=\tp_\Pi(\bb_2)
\end{cases}
$$
\end{lemma}
\begin{proof}
Let $\M$ and $\bb_1,\bb_2\in M^n$ be as stated, and assume $\tp_\Delta(\bb_1)=\tp_\Delta(\bb_2)$. It is enough to show that for any $\phi(x_0,...,x_{n-1})\in\Sigma$, $\M\models\phi(\bb_1)\iff\M\models\phi(\bb_2)$. Note that $\bb_1 = e_1[A_0]$ and $\bb_2=e_2[A_0]$ for some substructure $A_0$ of a finite $\Delta$-sub-model $\A$ of $T_0$. Without loss of generality, we will assume that $A_0 = \{a_0<\cdots<a_{n-1}\}$ is an initial segment of $A$.  For $\phi(x_0,...,x_{n-1})\in\Sigma$, define a coloring $h_\phi:\hom(\A,\M)\to 2$ by
$$h_\phi(e)=\begin{cases}
1 &\textrm{ if $\M\models\phi(e(a_0),...,e(a_{n-1}))$}\\
0 &\textrm{ if $\M\models\neg\phi(e(a_0),...,e(a_{n-1}))$}
\end{cases}$$
Now, as $T_0$ has the Ramsey property, there is an $\ell$-embedding $f:\M\qto\M$ such that $h_\phi$ is constant on $\hom(\A,f_X[M_X])$ for every $X\subset_\fin M$ -- that is, $h_\phi(fe)=h_\phi(fe')$ for all $e,e'\in\hom(\A,f_X[M_X])$, $X\subset_\fin M$. In particular, $h_\phi(f_Xe_1)=h_\phi(f_Xe_2)$ whenever $\bb_1\bb_2\subseteq X\subset_\fin M$, so 
\begin{align*}
\M\models\phi(f_X\bb_1) &\iff \M\models\phi(f_Xe_1(a_0),...,f_Xe_1(a_{n-1}))\\
&\iff \M\models\phi(f_Xe_2(a_0),...,f_Xe_2(a_{n-1}))\\
&\iff \M\models\phi(f_X\bb_2)
\end{align*}
Since $\M$ is $\Delta$-existentially closed, $f_X:M_X\to\M$ extends to a $\Delta$-embedding $\hat f_X:\M\to\M$. Again using the fact that $\M$ is e.c., we have $\M\models\phi(\hat f_X\bb_i)\iff \M\models\phi(\bb_i)$ ($i{=}1,2$); thus, $\M\models\phi(\bb_1)\iff \M\models\phi(\bb_2)$, as required.
\end{proof}

\begin{cor}
Let $T_0$ be an ordered theory of indiscernibles in a lanague $\L_0=\{<,...\}$. Then $T_0^*$ eliminates quantifiers down to $\Delta$. (Or in older parlance, $\Delta$ is an elimination set for $T_0^*$.) 
\end{cor}
\begin{proof}
Since $T_0^*$ is $\Delta$-model-complete, we already know that $T_0^*$ eliminates quantifiers down to $\Sigma$; thus, to prove the theorem, it is enough to show that for any $\phi(x_0,...,x_{n-1})\in\Sigma$, there is a $\theta_\phi(\xx)\in\Delta$ such that $T_0^*\models\phi(\xx)\bic \theta_\phi(\xx)$.  Let 
$$X_\phi =\left\{p\in S_n^\Delta(T_0):\textrm{$T_0^*\cup p(\xx)\cup\{\phi(\xx)\}$ is consistent}\right\}$$
Let $\theta_\phi(\xx)\in\Delta$ be such that $T_0\models \theta_\phi(\xx)\bic \bigvee_{p\in X_\phi}p(\xx)$, so that $T_0\models\phi(\xx)\cond \theta_\phi(\xx)$. It remains to show that $T_0^*\models\theta_\phi(\xx)\cond\phi(\xx)$.
Towards a contradiction, suppose $T_0^*\cup\{\neg\phi(\bb),\theta_X(\bb)\}$ is consistent -- let $(\M,\bb)$ be a model. Then $\M\models\neg\phi(\bb)$, but $\tp_\Delta(\bb)=p(\xx)\in X_\phi$. By definition of $X_\phi$ and the $\Delta$-model-completeness of $\M$, there is a tuple $\cc\in M^n$ such that $\M\models\phi(\cc)\wedge p(\cc)$. Hence, $\tp_\Delta(\bb) = \tp_\Delta(\cc)$ but $\tp_\Sigma(\bb)\neq\tp_\Sigma(\cc)$, contradicting the previous lemma (up to some fiddling with permutations of free variables). Thus, $T_0^*\models\theta_\phi(\xx)\bic\phi(\xx)$, as desired. 
\end{proof}

The next theorem now follows immediately from the previous corollary and the hypothesis that $S_n^\Delta(T_0)$ is finite for all $n$. 
\begin{thm}
Let $T_0$ be an ordered theory of indiscernibles  in a lanague $\L_0=\{<,...\}$. Then $T_0^*$ is $\aleph_0$-categorical and eliminates quantifiers down to its base $\Delta$. In particular, if the base is the set of quantifier-free formulas, then $T_0^*$ is the theory of the limit of \fraisse class of linearly-ordered finite structures. 
\end{thm}

\subsubsection{The Ramsey Property is as Strong as Possible}

\begin{defn}
Let $T_0$ be a Robinson theory in a language $\L_0=\{<,...\}$ with base $\Delta$. Suppose  $T_0^*$ is $\aleph_0$-categorical, and all models of $T_0^*$ are linearly order by $<$. We say that $T_0^*$ {\em almost-densely linearly ordered} if there is a 0-definable equivalence relation $E(x_1,x_2)$ such that for any model $\M$ of $T_0^*$,
\begin{enumerate}
\item For every $a\in M$, $[a]_E$ is finite. (Thus, the cardinality of $E$-classes is uniformly bounded.)
\item The relation $\leq_E\,=\left\{([a]_E,[b]_E): a\leq b\right\}$ is a dense linear order of $M/E$.
\end{enumerate}
\end{defn}

\begin{lemma}
Suppose $T_0$ is an ordered theory of indiscernibles in $\L_0 = \{<,...\}$ with base $\Delta$. Then $T_0^*$ is almost-densely linearly ordered.
\end{lemma}
\begin{proof}
For each $0<n<\omega$, let $\phi_n(x_1,x_2)$ be the formula 
$$\forall y_0...y_{n-1},y_n\left(\bigwedge_{i\leq n}x_1{<}y_i{<}x_2\cond \bigvee_{i<n}y_n=y_i\right)$$
asserting that there are at most $n$ distinct elements above $x_1$ and below $x_2$. Also, define $\phi_0(x_1,x_2)$ to be the formula $\forall y\neg (x_1<y<x_2)$. By $\aleph_0$-categoricity, there is an $0<N<\omega$ such that for all $N\leq n<\omega$
$$T_0^*\cup\{\neg\phi_N(x_1,x_2)\}\models \neg\phi_n(x_1,x_2)$$
Define $E(x_1,x_2)$ to be 
$$x_1=x_2\vee\bigvee_{0<n\leq N}\phi_n(x_1,x_2)\vee \bigvee_{0<n\leq N}\phi_n(x_2,x_1)$$
It is clear that $E$ is symmetric and reflexive, so to see that $E$ is an equivalence relation, it is enough to show that $E$ is transitive. For this, let $\M\models T_0^*$, and suppose $a_1,a_2,a_3\in M$ are such that $\M\models E(a_1,a_2)\wedge E(a_2,a_3)$. Without loss of generality, we may assume $a_1<a_2<a_3$. We consider the case that there are $b_0,...,b_{m-1}$, $b'_0,...,b'_{n-1}$ such that 
$$a_1<b_0<\cdots<b_{m-1}<a_2<b'_0<\cdots<b'_{n-1}<a_3$$
and $0<m,n\leq N$ are maximal. By choice of $N$, $m+n+1\leq N$, so $\M\models E(a_1,a_3)$. The cases wherein $a_2$ is the successor of $a_1$ and/or $a_3$ is the successor of $a_2$ may be treated similarly. 

Next, we show that $\leq_E$ is well-defined and a dense linear order. Given $[a]_E,[b]_E\in M/E$, define $a_0 = \min_<[a]_E$ and $a_1 = \max_<[a]_E$, and define $b_0,b_1$ similarly; in particular, $a_0\leq a\leq a_1$ and $b_0\leq b\leq b_1$. Suppose $[a]_E\leq_E[b]_E$ and $[b]_E\leq_E[a]_E$. Without loss of generality, assume $a_1\leq b_1$; then $b_0\leq b\leq a\leq a_1\leq b_1$, so $a\in [b]_E$. Hence $\leq_E$ is well-defined. It is easy to see that $\leq_E$ is a partial order of $M/E$, so we need only show that $<_E$ is the trichotomy and that it is dense. Assume $[a]_E\neq [b]_E$; then $a_0,a_1\notin [b]_E$, so either $a_1<b_0$ or $b_1<a_0$. This suffices for the trichotomy. For density, suppose $[a]_E\,<_E\,[b]_E$, so that $a_1<b_0$. If there is no $c\in (a_1,b_0)$, then $E(a_1,b_0)$, which is impossible. Hence, we may choose any $c\in (a_1,b_0)$ with the guarantee that $[a]_E<_Ec/E<_E[b]_E$.
\end{proof}

\begin{defn}
Let $T_0$ be an ordered theory of indiscernibles with base $\Delta$ in $\L_0=\{<,...\}$, and assume that $T_0$ is almost-dense via $E$. We say that $T_0$ is {\em strongly 1-sorted} if for all $a,b\in M$, $\M\models T_0^*$, if $a,b\notin\acl(\emptyset)$, then 
$$\tp_\Delta([a]_E) = \tp_\Delta([b]_E)$$
\end{defn}

\begin{lemma}
Let $T_0$ be strongly 1-sorted ordered theory of indiscernibles with base $\Delta$ in $\L_0=\{<,...\}$, and assume that $T_0$ is almost-dense via $E$. Suppose $\M\models T_0^*$, and let $A,B\subset_\fin M$ and $c\in M$ be such that $\acl(A)\cap\acl(B)=\acl(\emptyset)$ and $\acl(c)\cap\acl(AB)=\acl(\emptyset)$. Moreover, suppose that for all $a\in \acl(A)$
$$[c]_E<_E[a]_E\,\,\implies a\in\acl(\emptyset).$$
Then,
\begin{enumerate}
\item There is a subset $B'\subset_\fin M$ such that $B'\equiv_AB$ and for all $b'\in \acl(B')$
$$[c]_E<_E[b']_E\,\,\implies b'\in\acl(\emptyset)$$

\item There is a subset $B''\subset_\fin M$ such that $B''\equiv_AB$ and for all $b''\in \acl(B'')$
$$[b'']_E<_E[c]_E\,\,\implies b'\in\acl(\emptyset)$$
\end{enumerate}
\end{lemma}
\begin{proof}
We expose the proof part 1 of the claim; the proof of part 2 is almost identical. Let $\aa_0\aa$, $\aa_0\bb$, and $\aa_0\cc$ be enumerations of $\acl(A)$, $\acl(B)$, and $\acl(c)$, respectively, where $\aa_0$ is an enumeration of $\acl(\emptyset)$. Hence, $\aa$,$\bb$,$\cc$ are pairwise disjoint as subsets of $M$. Let $p_B(\xx,\aa,\aa_0) = \tp_\Delta(\bb,\aa,\aa_0)$ where $\xx=(x_0,...,x_{m-1})$, and let $p_c(\yy,\aa,\aa_0)=\tp_\Delta(c,\aa,\aa_0)$ where $\yy = (y_0,...,y_{n-1})$. Without loss of generality, we assume $c_0 = \min\cc$ and $b_{n-1} = \max\bb$.
 Taking $\aa,\aa_0,\bb,\cc$ as new constant symbols, the negation of claim the implies,
$$T_0^*\models p_c(\cc,\aa,\aa_0)\wedge p_B(\bb,\aa,\aa_0)\cond c_0<b_{n-1}$$
By the Craig Interpolation Theorem, there is sentence $\theta(b_{n-1},c_0)$ such that 
$$T_0^*\models p_c(\cc,\aa,\aa_0)\wedge p_B(\bb,\aa,\aa_0)\cond\theta(b_{n-1},c_0)$$
$$T_0^*\models \theta(b_{n-1},c_0)\cond  c_0<b_{n-1}$$
By further manipulations from basic logic, we may assume that $\theta(x,y)=\theta_B(x)\wedge\theta_c(y)$ where $\theta_B=\tp_\Delta(b_{n-1})$, $\theta_c=\tp_\Delta(c_0)$. From the assumption that $T_0$ is strongly 1-sorted, $\tp_\Delta([c_0]_E) = \tp_\Delta([b_{n-1}]_E)$, and it follows that $T_0^*\models \forall x(x\notin \acl(\emptyset)\cond x<x)$, which is nonsense. 
 \end{proof}


The proof of the next theorem is an adaptation of a proof of the fact that $\QQ\not\to(\QQ)^2_2$, which is actually a special case of this theorem. The proof of the special case was communicated to us by L. Scow. 
\begin{thm}
Let $T_0$ be a strongly 1-sorted ordered theory of indiscernibles with base $\Delta$ in $\L_0=\{<,...\}$, and assume that $T_0$ is almost-dense via $E\in\Delta$. Let $\M\models T_0$ be a countable model. There, a finite model $\A\prec_\Delta\M$ of $T_0$ and a coloring $h:\hom(\A,\M)\to 2$ such that for every $\Delta$-embedding $f:\M\to\M$, there are $e,e'\in\hom(\A,f\M)$ such that $h(e)\neq h(e')$.
\end{thm}
\begin{proof}
Set $M_0 = M\setminus\acl(\emptyset)$, and let $u:\QQ{\cap}(0,1)\to (M_0/E,<)$ be an order-isomorphism. For each $2\leq n<\omega$, define 
$$F_n = \left\{\frac{k}{n}: k\in\{1,2,...,n-1\},\mathrm{gcd}(k,n)=1\right\}$$
Then, $\QQ\cap(0,1) = \bigcup_{2\leq n<\omega}F_n$, and we may define $\rk:\QQ\cap(0,1)\to \omega$ by, $\rk(q) = n\iff q\in F_n$. Finally, define an auxiliary ordering $<_F$ on $\QQ\cap(0,1)$ by 
$$q<_Fr\iff \bigvee\begin{cases}
\rk(q)<\rk(r)\\
(\rk(q)=\rk(r)\wedge q<r)
\end{cases}$$
Let $\A\prec_\Delta\M$ be a finite model of $T_0$ such that $\acl(\emptyset)\subset A$ and there are $a_0,a_2\in A\setminus \acl(\emptyset)$ such that 
\begin{enumerate}
\item $[a_i]_E\subseteq A$ for each $i<2$;
\item $\acl(a_0)\cap\acl(a_i)=\emptyset$;
\item $[a_0]_E$ is the minimum element of $(A\setminus\acl(\emptyset))/E$, and $[a_1]_E$ is the maximum element of $(A\setminus \acl(a_0))/E$.
\end{enumerate}
Now, we define a coloring $h:\hom(\A,\M)\to 2$ by 
$$h(e) = \begin{cases}
0&\textrm{if $u^{-1}\big([e(a_0)]_E\big) <_Fu^{-1}\big([e(a_1)]_E\big)$}\\
1&\textrm{if $u^{-1}\big([e(a_1)]_E\big) <_Fu^{-1}\big([e(a_0)]_E\big)$} 
\end{cases}$$
Without loss of generality, we may assume that $a_0<a_1$ (all that is really important is that $a_0,a_1$ are in fixed positions in the order of $A$).

We will show that there is \emph{no} $\Delta$-embedding $f:\M\to\M$ such that $h$ is constant on $\hom(\A,f\M)$.
Suppose $f:\M\to\M$ is a $\Delta$-embedding. Set  
$$W_\A = \big\{\{q<r\}\in[\QQ]^{n+1}:u(q)u(r)\equiv^\Delta [a_0a_1]_E\big\}$$
Now, we fix $\{q<r\}\in W_\A$ such that $u(q)=[f(a_0)]_E$ and $u(r)=[f(a_1)]_E$. Clearly, there are only two possibilities to consider:
\begin{itemize}
\item \textbf{Case}: $q<_Fr$:

Assume $r\in F_n$ for some $n<\omega$.  By the previous lemma, there are infinitely many $\Delta$-embeddings $e_i:\A\to\M$ ($i<\omega$) such that $[f(a_0)]_E<_E[e_{i+1}(a_0)]_E<_E[e_i(a_0)]<_E [f(a_1)]_E$ and $[e_i(a_1)]_E = [f(a_1)]_E$ for all $i<\omega$. Since $F_n$ is finite, there is an $i<\omega$ such that $u^{-1}\big([e_i(a_0)]_E\big)>r$.

\item \textbf{Case}: $r<_Fq$:

Assume $q\in F_n$ for some $n<\omega$.  By the previous lemma, there are infinitely many $\Delta$-embeddings $e_i:\A\to\M$ ($i<\omega$) such that $[f(a_1)]_E<_E[e_{i}(a_1)]_E<_E[e_{i+1}(a_1)]$ and $[e_i(a_0)]_E = [f(a_0)]_E$ for all $i<\omega$. Since $F_n$ is finite, there is an $i<\omega$ such that $u^{-1}\big([e_i(a_1)]_E\big)>q$
\end{itemize}
In either case, we see that $h$ is not constant on $\hom(\A,f\M)$, as required.
\end{proof}

\section{Extreme amenability}

The connection between Ramsey classes and extremely amenable groups proved in \cite{kpt-2005} is one of the major accomplishments in structural Ramsey theory. In this section, we provide what we hope is a somewhat more accessible proof of that result, using our framework and terminology to streamline some arguments. (Other than these changes, there is really nothing new in this section.)

\begin{defn}
Let $M$ be a countably infinite set. The topology on $\mathrm{Sym}(M)$ -- the group of permutations of $M$ -- has basic open sets,
$$U_{\aa,\bb} = \left\{g\in\mathrm{Sym}(X):g\aa = \bb\right\}$$
where $\aa,\bb\in M^n$, $n<\omega$. A subgroup $G\leq \mathrm{Sym}(M)$ is closed just in case it is a closed set in this topology. 
If $G$ is a closed group and $X$ some other topological space, a \emph{continuous action of $G$ on $X$} is a continuous function $\psi:G\times X\to X$ that also happens to define a group action of $G$ on $X$.

Finally, $G$ is \emph{extremely amenable} if every continuous action $G\times X\to X$ on a compact Hausdorff space $X$ has a fixed point -- an element $x\in X$ such that $g(x)  = x$ for all $g\in G$.
\end{defn}

\begin{prop}
Let $T_0$ be an ordered Robinson theory in $\L_0$ with base $\Delta$ such that $T_0^*$ is $\aleph_0$-categorical. Assume that $\Delta$ is an elimination set for $T_0^*$ and that $T_0$ has finite-rigidity in the sense of Lemma \ref{lemma:finite-rigidity}. Let $\M$ be a countable model of $T_0^*$. If $Aut(\M)$ is extremely amenable, then $T_0^*$ has the Ramsey property. 
\end{prop}
\begin{proof}
 Let $\A\models T_0$ be a finite structure, and suppose $h:\hom(\A,\M)\to k$ is a coloring. By the finite-rigidity assumption, we identify $\A$ with an $n$-tuple $\aa^*\in M^n$. Let $\L_0'$ be the expansion of $\L_0$ with new $n$-ary predicates $P_0,...,P_{k-1}$, and let $\M'$ be the expansion of $\M$ in which 
$$P_i^{\M'} = \left\{\bb\in M^n:\tp_\Delta(\bb)=\tp_\Delta(\aa^*),\, h(\bb)=i\right\}$$
Then $X = \left\{p\in S_x^{\L_0'}(M):\{x\neq m\}_{m\in M}\subset p\right\}$ is a compact Hausdorff space.

\begin{claim}
$Aut(\M)\times X\longrightarrow X:(g,p)\mapsto\left\{\phi(x,g\bb):\phi(x,\bb)\in p\right\}$ is a continuous action.
\end{claim}
\begin{proof}[Proof of claim]
Let $\phi(x,\bb)\in \L'_0(M)$. We need only show that the pre-image of a basic open set $[\phi(x,\bb)]$, 
$$\left\{(g,p): \phi(x,\bb)\in g(p)\right\}=\left\{(g,p): \phi(x,g^{-1}\bb)\in p\right\}$$
is open. Let $g\in Aut(\M)$. The product $U_{g^{-1}\bb,\bb}\times[\phi(x,g^{-1}\bb)]$ is an open subset of $Aut(\M)\times S_1^{\L_0'}(M)$. Now, 
$$\left\{(g,p): \phi(x,g^{-1}\bb)\in p\right\}=\bigcup_{g\in Aut(\M)}U_{g^{-1}\bb,\bb}\times[\phi(x,g^{-1}\bb)]$$
is an open set, as required.
\end{proof}
By extreme amenability, this action has a fixed-point, say $p$. We will use $p$ to construct an $\ell$-embedding $f:\M\qto\M$ such that $h$ is constant on $\hom(\A,f_X[M_X])$ for every $X\subset_\fin M$. 
 Let $X\subset_\fin M$ be given. We take $M_X\prec_\Delta \M$ to be any finite model of $T_0$ such that $X\aa^*\subseteq M_X$ and $\hom(\A,M_X)$ is non-empty. Also, we choose a $b\in M\setminus M_X$ and an automorphism $g_X\in Aut(\M)$ such that $\tp(b/g_X[M_X])\subset p(x)$. 

We claim now that $h$ is constant on $\hom(\A,g_X[M_X])$. Let $\aa_1,\aa_2\in g_X[M_X]^n$ be a realizations of $\tp_\Delta(\aa^*)$. Since $\Delta$ is an elimination set for $T_0^*$ and $T_0^*$ is $\aleph_0$-categorical (so that $\M$ is $\aleph_0$-homogeneous), there is an automorphism $g\in Aut(\M)$ such that $g\aa_1=\aa_2$. Since $p$ is a fixed point -- so that $g(p) = p$ -- it follows that $h(\aa_1) = h(\aa_2)$. Thus, it suffices to define $f_X = g_X\r M_X$.
This completes the proof of the proposition.
\end{proof}

The remainder of this section is really quite close to the presentation in \cite{kpt-2005}. The changes are mainly notational, except that all references to $\mathbb{R}$ have been eliminated, and our framework for the Ramsey property simplifies some parts of these arguments.

\begin{lemma}\label{lemma:top-criterion}
Let $G\subseteq \mathrm{Sym}(M)$ be a closed group, and let $\psi:G\times X\to X$ be a continuous action on a compact Hausdorff space. The following are equivalent:
\begin{enumerate}
\item $\psi$ has a fixed-point.
\item For every finite $G_0\subseteq_\fin G$ and every finite closed cover $\mathcal{K}$ of $X$, there are $x\in X$ and $K\in\mathcal{K}$ such that $\{\psi_g(x):g\in G_0\}\subseteq K$. (Below, we denote this condition by $\maltese_{\psi}$.)
\end{enumerate}
\end{lemma}
\begin{proof}
1$\implies$2 is immediate, so we move directly to the proof of 2$\implies$1. Let $\psi:G\times X\to X$ be a continuous action on a compact Hausdorff space. For any finite $G_0\subseteq_\fin G$ and any finite closed cover $\mathcal{K}$ of $X$, we define,
$$C(G_0,\mathcal{K}) = \Big\{x\in X: \exists K\in\mathcal{K}. \{\psi_g(x):g\in G_0\}\subseteq K\Big\}.$$
\begin{claim}
For every finite $G_0\subseteq_\fin G$ and every finite closed cover $\mathcal{K}$ of $X$, $C(G_0,\mathcal{K})$ is closed.
\end{claim}
\begin{proof}[Proof of claim]
Let $g\in G$, and let $\mathcal{K} = \{K_0,...,K_{N-1}\}$ be a closed cover. For each $i<N$, $K_i\cap g^{-1}K_i$ is a closed set. For a finite set $G_0\subseteq_\fin G$, it follows that 
$$C(G_0,\mathcal{K}) = \bigcup_{i<N}\bigcap_{g\in G_0}(K_i\cap g^{-1}K_i)$$
is closed.
\end{proof}
\begin{claim}
$\bigcap_{G_0,\mathcal{K}}C(G_0,\mathcal{K})\neq\emptyset$
\end{claim}
\begin{proof}[Proof of claim]
As $X$ is compact, every $C(G_0,\mathcal{K})$ is compact, so it is enough to show that the intersection $\bigcap_{i<t}C(G_t,\mathcal{K}_t)$ is non-empty, whenever $G_0,...,G_{t-1}\subseteq_\fin G$ and $\mathcal{K}_0,...,\mathcal{K}_{t-1}$ are finite closed covers. Set $H = G_0\cup \cdots\cup G_{t-1}$, and define $\mathcal{K}$ to be the family consisting of the non-empty sets of the form 
$$K_{0,i_0}\cap\cdots\cap K_{t-1,i_{t-1}}$$
where $K_{j,i_j}\in\mathcal{K}_j$ for each $j<N$. By hypothesis, $C(H,\mathcal{K})$ is non-empty, and clearly, $C(H,\mathcal{K})\subseteq\bigcap_{i<t}C(G_t,\mathcal{K}_t)$, as required.
\end{proof}
\noindent Finally, let $x_0\in \bigcap_{G_0,\mathcal{K}}C(G_0,\mathcal{K})$. We claim that $x_0$ is a fixed-point of $G$. Let $g\in G$. If $\psi_g(x_0)\neq x_0$, then as $X$ is Hausdorff, there are open sets $U_0,U_1$ such that $x_0\in U_0$, $\psi_g(x_0)\in U_1$ and $U_0\cap U_1\neq\emptyset$. Let $K_0 = X\setminus U_1$, $K_1 = X\setminus U_0$, and $\mathcal{K} = \{K_0,K_1\}$ -- a finite closed cover of $X$. By definition, $x_0\in C(\{g\},\mathcal{K})$, but this is impossible. Thus, $x_0$ is a fixed-point, as claimed.
\end{proof}

\begin{lemma}\label{lemma:top-criterion-2}
Let $T_0$ be a theory of indiscernibles, and let $\M\models T_0^*$ be a countable model. Then $G = Aut(\M)$ has the following property:
\begin{quote}
($\star$) Suppose $G_0\subseteq_\fin G$, $\A\prec_\Delta\M$ is finite, and $\xi:G/G_A\to k$ is a coloring for some $k<\omega$. Then there is a $g^*\in G$ such that 
$$\xi\left(g^*g_1G_A\right) = \xi\left(g^*g_2G_A\right)$$
for all $g_1,g_2\in G_0$.
\end{quote}
\end{lemma}
\begin{proof}[Proof (Sketch)]
Let $\aa$ be an enumeration of $A$ in the $\Delta$-definable ordering of $T_0$. We then arrive at the following bijective correspondences,
$$G/G_A\overset{\sim}{\longleftrightarrow}\aa^G\overset{\sim}{\longleftrightarrow}\hom(\A,\M)$$
in the natural way. Moreover, $G_0$ may be identified with a finite set $\{A_0,...,A_{n-1}\}$ of copies of $\A$; set $X^* = \bigcup_{i<n}A_i$. Finally, the coloring $\xi:G/G_A\to k$ is ``really'' a coloring $\xi:\hom(\A,\M)\to k$. By the Ramsey property of $T_0$, there is an $\ell$-embedding $f:\M\qto\M$ such that $\xi$ is constant on $\hom(\A,f_X[M_X])$ for all $X\subset_\fin M$. In particular, $\xi$ is constant on $\hom(\A,f_{X^*}[M_{X^*}])$, and translating back through our correspondences yields the desired automorphism $g^*\in G$.
\end{proof}

\begin{lemma}\label{lemma:group-criterion}
Let $T_0$ be a Robinson theory, and let $\M\models T_0^*$. If $G = Aut(\M)$ has property ($\star$) as in the previous lemma, then $G$ is extremely amenable.
\end{lemma}
\begin{proof}[Proof (Sketch)]
 We prove the result for right cosets, but this is sufficient. It is enough to verify that the second part of Lemma \ref{lemma:top-criterion} is satisfied, so let $\psi:G\times X\to X$ be a continuous action on a compact Hausdorff space, $\mathcal{K}$ a closed cover of $X$, and $G_0\subset_\fin G$. Let $\A\prec_\Delta\M$ be a finite model of $T_0$ -- then $G_A$ is an open subgroup of $G$. We fix an $x_0\in X$ arbitrarily, and for each $K\in\mathcal{K}$, we define,
 $$U_K = \left\{g\in G: \psi_g(x_0)\in X\setminus (\cup_{K'\neq K}K')\right\}$$
 an open set, and,
 $$V_K = G_A\cdot U_K = \bigcup_{h\in U_K}hG_A.$$
 Clearly, $\bigcup_{K\in\mathcal{K}}V_K = G/G_A$. We may now choose a coloring $\xi:G/G_A\to \mathcal{K}$ such that $\xi^{-1}(K)\subseteq V_K$ for every $K\in\mathcal{K}$. By property ($\star$), there are $h\in G$ and $K_0\in \mathcal{K}$ such that $\xi(G_Agh^{-1}) = K_0$ for all $g\in G_0$. It is not difficult, then, to verify that 
 $$\left\{\psi_{gh^{-1}}(x_0):g\in G_0\right\}\subseteq K_0$$
 so $x_1 = \psi_{h^{-1}}(x_0)$ suffices.
\end{proof}

\begin{prop}
Let $T_0$ be a theory of indiscernibles, and let $\M\models T_0^*$ be a countable structure. Then $Aut(\M)$ is extremely amenable.
\end{prop}
\begin{proof}
Apply Lemmas \ref{lemma:top-criterion-2} and \ref{lemma:group-criterion}.
\end{proof}

As a last remark for this section (also observed in \cite{bodirsky-pinsker-2011}), we note that by simultaneously characterizing the Ramsey property both in terms of theories and automorphism groups, the following theorem now follows as an easy consequence of a classical result from \cite{az-1986}.

\begin{thm}
Let $T_0$, $T_1$ be Robinson theories in $\L_0,\L_1$, respectively. If $T_0^*$ and $T_1^*$ are bi-interpretable, then $T_0$ is a theory of indisernibles if and only if $T_1$ is a theory of indiscernibles.
\end{thm}

\section{Two Examples}


\subsubsection*{Convex equivalence relations}

It is relatively well-known that the class of ``convex equivalence relations'' forms a Ramsey class. (A member of this class is a finite linear order equipped with an equivalence relation whose classes are convex.) In this subsection, we convex equivalence relations as an example demonstrating that, in principle, the unconstrained modeling property can be used to \emph{prove} a Ramsey property for a class of finite structures, rather than the other way around.
Let $\L_0$ be the language built over $\{E^{(2)},\leq \}$. Let $T_0$ be the following theory:
\begin{enumerate}
\item $\forall xyz\left(E(x,y)\wedge E(y,z)\cond E(x,z)\right)$

$\forall xy\left(E(x,y)\cond E(y,x)\right)$

$\forall x\, E(x,x)$

\item $\forall xyz(x\leq y\wedge y\leq z\cond x\leq z)$

$\forall xy\left(x\leq y\wedge y\leq x\cond x=y\right)$

$\forall x\,x\leq x$

\item $\forall xyz\left(x\leq y\leq z\wedge E(x,z)\cond E(x,y)\right)$
\end{enumerate}
In this case, the base of $T_0$ is just the set of quantifier-free $\L_0$-formulas -- the trivial base. Now, up to isomorphism, we can explicitly describe the countable model $\A$ of $T_0^*$.
Let $\CC\subset\mathbb{R}$ denote the middle-thirds Cantor set -- that is, 
$$\CC = [0,1]\setminus \bigcup_{0<n<\omega}\left(\bigcup_{0\leq k<3^{n-1}}\left( \frac{3k+1}{3^n}, \frac{3k+2}{3^n}  \right)\right)$$
The universe of $\A$ is $\QQ\cap\left([0,1]\setminus \CC\right)$; $<^\A$ is the linear order induced by that of $\mathbb{R}$; and the classes of $E^\A$ are precisely the sets $\QQ\cap \left( \frac{3k+1}{3^n}, \frac{3k+2}{3^n}  \right)$ for all $0<n<\omega$ and $0\leq k<3^{n-1}$.

\begin{prop}
$T_0$ has the unconstrained modeling property.
\end{prop}
\begin{proof}[Proof (Sketch)]
By compactness, it is enough to show that for any $T_0$-template $(\A,\Gamma)$ into a structure $\M$, where $\A$ is the countable model of $T_0^*$, there is an  indiscernible $T_0$-picture $(\A,g)$ in a model $\M_1\succeq\M$ patterned on $(\A,\Gamma)$. Again, by compactness, we just need to show that the set of sentences $\mathtt{Scow}$ of the ``only if'' part of Theorem \label{thm:main-equivalence} is finitely satisfiable, and for this, given the diagram of a finite substructure of $\A$, one just applies the finitary Ramsey theorem for (unadorned) linear orders several times with a sufficiently large finite extension.
\end{proof}

\begin{cor}
Let $K$ be the \fraisse class of finite linear orders equipped with a convex equivalence relation. $K$ is a Ramsey class.
\end{cor}

\subsubsection*{Trees}

Let $\L_0$ be the language with two binary predicate symbols $\leq,\,\tpo$, a constant symbol $\mathbf{0}$, and two ternary predicate symbols $R$ and $[\cdot,\cdot\,\|\,\cdot]$. 
Let $T_0$ be the theory that expresses the following:
\begin{enumerate}
\item ``$\tpo$ is a partial order with least element $\mathbf{0}$.''

``$\tpo$ is tree-like'': $\forall x\forall y_1y_2\left[y_1\tpo x\wedge y_2\tpo x\cond(y_1\tpo y_2\vee y_2\tpo y_1)\right]$

\item ``$\leq$ is a linear order''

$\forall xy(x\tpo y\to x\leq y)$

\item $R$ defines the most-recent-common-ancestor relation:


$$\forall xy_1y_2\left(R(x,y_1,y_2)\cond \bigwedge\begin{cases}
x\tpo y_1\wedge x\tpo y_2\wedge \neg(y_1\tpo y_2\vee y_2\tpo y_1)\\
\forall x'(x'\tpo y_1\wedge x'\tpo y_2\cond x'\tpo x)\\
R(x,y_1,y_2)\\
\end{cases}\right)$$
$$\forall xx'y_1y_2\left(R(x,y_1,y_2)\wedge R(x',y_1,y_2) \cond x=x'\right)$$ 
$$\forall xy_1y_1'y_2y_2'\big[R(x,y_1,y_2)\wedge y_1\tpo y_1'\wedge y_2\tpo y_2'\cond R(x,y_1',y_2')\big]$$
$$\forall xy_1y_1'y_2y_2'\big[R(x,y_1,y_2)\wedge \bigwedge_ix\triangleleft y_i'\tpo y_i\cond R(x,y_1',y_2')\big]$$
$$\forall y_1y_2(y_1\neq y_2\cond \exists x\, R(x,y_1,y_2))$$

\item $[x_1,x_2\|x_3]$ describes the arrangement,
$$
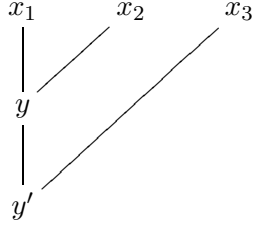
\xymatrix{
x_1 & x_2 & x_3 \\
y\ar@{-}[u]\ar@{-}[ur] & & \\
y'\ar@{-}[u]\ar@{-}[uurr] & & 
}$$
Formally, this is given by the axioms,
$$\forall x_1x_2x_3\big([x_1,x_2\|x_3]\cond x_1\neq x_2\wedge \forall y,y'\left(R(y,x_1,x_2)\wedge R(y',x_1,x_3)\cond y'\triangleleft\, y\right)\big)$$
$$\forall x_1x_2x_3\big([x_1,x_2\|x_3]\cond [x_2,x_2\|x_3]\big)$$

\item 

$$\forall x_1x_2x_3\big([x_1,x_2\|x_3]\wedge x_1<x_2\cond (x_3<x_1\vee x_2<x_3\big)$$

%
%
%
\end{enumerate}

\begin{obs}
Suppose $\A = (A,\tpo^\A,\leq^\A,R^\A,[\cdot,\cdot\|\cdot],\mathbf{0}^\A)$ is a \emph{finite} model of $T_0$. Then for all $a_1,a_2\in A$, if $\mathbf{0}^\A\neq a_1$ and $a_2$ is the $\leq^\A$-successor of $a_1$, then either $a_1\tpo^\A a_2$ or $a_2$ and $a_1$ are siblings (share a common $\tpo^\A$-predecessor).
\end{obs}

It isn't hard to verify that the class of finite models of $T_0$ is actually a \fraisse class. In fact, one can fairly easily construct the countable finitely-universal model of $T_0$ explicitly as follows:
\begin{enumerate}
\item Given a model $\A$ of $T_0$, let $\A'$ be the model of $T_0$ obtained by ``gluing'' a copy of ${^{<\omega}}\omega$ to a new node $x$ between $\sigma$ and $\sigma\widehat{\,\,}n$ for all $\sigma\in {^{<\omega}}\omega$ and $n<\omega$. 
\item Let $\A_0$ be ${^{<\omega}}\omega$ itself as a model of $T_0$, and for all $n<\omega$, let $\A_{n+1} = \A_n'$.
\end{enumerate}
 Then, $\A = \bigcup_n\A_n$ is the \fraisse limit of the finite models, and $T^*_0 = Th(\A)$ is model-companion of $T_0$. By a theorem of \cite{milliken-1979} -- reproved in \cite{bodirsky-piguet-2010} -- the $T_0$ has the finitary Ramsey property, so $T_0$ is a theory of indiscernibles.

\bibliographystyle{plain}
\bibliography{myref}

\begin{thebibliography}{1}

\bibitem{az-1986}
G.~Ahlbrandt and M.~Ziegler.
\newblock Quasifinitely axiomatizable totally categorical theories.
\newblock {\em Annals of Pure and Applied Logic}, 30:63--82, 1986.

\bibitem{bodirsky-piguet-2010}
Manuel Bodirsky and Diana Piguet.
\newblock Finite trees are ramsey under topological embeddings.
\newblock {\em Unpublished technical report}, 2010.

\bibitem{bodirsky-pinsker-2011}
Manuel Bodirsky and Michael Pinsker.
\newblock {\em Model Theoretic Methods in Finite Combinatorics}, volume 558 of
  {\em AMS Contemporary Mathematics}, chapter ``Reducts of Ramsey structures''.
\newblock American Mathematical Society, 2011.

\bibitem{hru-1997}
Ehud Hrushovski.
\newblock Simplicity and the lascar group.
\newblock {\em Preprint}, 1997.

\bibitem{hyttinen-2000}
Tappani Hyttinen.
\newblock On stability in finite models.
\newblock {\em Archive for Mathematical Logic}, 39:89--102, 2000.

\bibitem{kpt-2005}
A.~Kechris, V.~Pestov, and S.~Todorcevic.
\newblock Fra\"iss\'e limits, ramsey theory, and topological dynamics of
  automorphism groups.
\newblock {\em Geometric and Functional Analysis}, 15(1):106 -- 189, 2005.

\bibitem{kim-kim-scow-2012}
Byunghan Kim, Hyeung-Joon Kim, and Lynn Scow.
\newblock Tree indiscernibilities, revisited.
\newblock {\em (submitted)}, 2012.

\bibitem{milliken-1979}
K.~R. Milliken.
\newblock A ramsey theorem for trees.
\newblock {\em Journal of Combinatorial Theory, Series A}, 26(3):215 --- 237,
  1979.

\bibitem{scow-2011}
Lynn Scow.
\newblock Characterization of nip theories by ordered graph-indiscernibles.
\newblock {\em Annals of Pure and Applied Logic}, December 2011.

\end{thebibliography}

\end{document}